\documentclass{amsart}
\usepackage{amsmath}
\usepackage[T1]{fontenc}
  \usepackage{paralist}
  \usepackage{graphics} 
  \usepackage{epsfig} 
 \usepackage[colorlinks=true]{hyperref}
\hypersetup{urlcolor=blue, citecolor=red}

\newtheorem{theorem}{Theorem}[section]
\newtheorem{corollary}[theorem]{Corollary}

\newtheorem{lemma}[theorem]{Lemma}
\newtheorem{proposition}[theorem]{Proposition}

\theoremstyle{definition}
\newtheorem{definition}[theorem]{Definition}
\newtheorem{remark}[theorem]{Remark}
\newtheorem*{notation}{Notation}

\newcommand{\R}{{\mathbb{R}}}

\newcommand{\Z}{{\mathbb{Z}}}
\newcommand{\N}{{\mathbb{N}}}

\newcommand{\xbm}{(X,\mathcal{B},\mu)}

\newcommand{\vep}{\varepsilon}
\newcommand{\bez}{\nopagebreak\hspace*{\fill}\nolinebreak$\Box$}

\title[Disjointness of interval exchange
transformations]{Disjointness of interval exchange transformations
from systems of probabilistic origin}

\author[Jacek Brzykcy and Krzysztof Fr\k{a}czek]{}



\email{eddie@mat.umk.pl} \email{fraczek@mat.umk.pl}

\thanks{Research partially supported by MNiSzW grant N N201
384834 and Marie Curie "Transfer of Knowledge" program, project
MTKD-CT-2005-030042 (TODEQ)}

\begin{document}
\maketitle

\centerline{\scshape Jacek Brzykcy}
\medskip
{\footnotesize
 \centerline{Faculty of Mathematics and Computer Science}
   \centerline{Nicolaus
Copernicus University}
   \centerline{ ul. Chopina 12/18, 87-100 Toru\'n, Poland}}

\medskip

\centerline{\scshape  Krzysztof Fr\k{a}czek }
\medskip
{\footnotesize
 \centerline{Faculty of Mathematics and Computer Science}
   \centerline{Nicolaus
Copernicus University}
   \centerline{ ul. Chopina 12/18, 87-100 Toru\'n, Poland}
 \medskip
  \centerline {Institute of Mathematics,
Polish Academy of Science} \centerline {ul. \'Sniadeckich 8,
00-956 Warszawa, Poland }}

\bigskip

\begin{abstract} It is proved that almost every interval exchange
transformation given by  the symmetric permutation
$$\begin{pmatrix}
    1&2&\ldots&m-1&m\\
    m&m-1&\ldots&2&1\\
        \end{pmatrix},$$
where $m\geq 5$ is an odd number, is disjoint from ELF systems.
Some disjointness properties  of special flows built over interval
exchange transformations and under piecewise constant roof
function are investigated as well.
\end{abstract}

\section{Introduction}
The notion of ELF systems was introduced in paper \cite{Fr-Le1} to
express the fact that a given system is of probabilistic origin.
An automorphism $T:\xbm\to\xbm$ is said to has the ELF property,
if it is ergodic and the weak closure of the set of all its
iterations $\{T^n:n\in\N\}$, considered as Markov operators in
$L^2\xbm$, consists of indecomposable Markov operators. The
following standard classes of systems of probabilistic origin
enjoy the ELF property: mixing systems, ergodic Gaussian systems
(see \cite{Le-Pa-Th} and \cite{Fr-Le1}), Poisson suspensions,
dynamical systems coming from stationary symmetric
$\alpha$--stable processes (see \cite{De-Fr-Le-Pa}) and infinitely
divisible processes (see \cite{Roy}).

This work continues the research programme  begun in
\cite{Fr-Le1}. The purpose of that project is to study
deterministic systems that are disjoint from systems of
probabilistic origin. This is closely related to some problems of
smooth realization for systems of probabilistic origin (see
\cite{Fr-Le2}). Recall that two measure-preserving automorphisms
of standard probability spaces are disjoint if they have only one
joining equal to the product measure (see \cite{Fu}). According to
the programme, in \cite{Fr-Le2}, it was proved that smooth ergodic
flows of compact orientable smooth surfaces having  only
non--degenerate saddles as isolated critical points (and having a
``good'' transversal) are disjoint from  ELF flows. A substantial
contribution to the project was provided in \cite{De-Fr-Le-Pa}
where some classes of automorphisms disjoint from ELF
automorphisms were found. The proofs of disjointness in
\cite{Fr-Le1} and \cite{Fr-Le2} were based mainly on the following
two results.
\begin{proposition}[\cite{Fr-Le2}]\label{proprozlflow} Suppose that
$\mathcal{T}=\{T_t\}_{t\in\mathbb{R}}$ is an ergodic flow on
$(X,\mathcal{B},\mu)$ such that there exist a sequence
$\{t_n\}\subset\mathbb{R}$, $t_n\to+\infty$, $0<\alpha\leq1$,
$J\in\mathcal{J}(\mathcal{T})$ and a  probability Borel measure on
$\mathbb{R}$ for which
$$T_{t_n}\rightarrow\alpha\int_{\mathbb{R}}T_sdP(s)+(1-\alpha)J\text{ weakly in }\mathcal{L}(L^2(X,\mathcal{B},\mu)).$$
If $P$ is not a Dirac measure then $\mathcal{T}$ is disjoint from
all weakly mixing ELF flows.
\end{proposition}

\begin{proposition}[\cite{De-Fr-Le-Pa}] \label{proprozaut} Let
$T:(X,\mathcal{B},\mu)\rightarrow(X,\mathcal{B},\mu)$ be a weakly
mixing ergodic automorphism. Assume that there exist an increasing
sequence $\{t_n\}$ of natural numbers, $0<\alpha\leq 1$,
$J\in\mathcal{J}(T)$ and a probability measure $P$ on $\Z$ such
that
$$T^{t_n}\rightarrow\alpha\sum_{k\in\Z}P(\{k\})T^k+(1-\alpha)J\text{ weakly in }\mathcal{L}(L^2(X,\mathcal{B},\mu)).$$
If $P$ is not Dirac then $T$ is disjoint from all ELF
automorphisms.
\end{proposition}
\noindent It was shown in \cite{De-Fr-Le-Pa} that weakly mixing
simple but not mixing automorphisms and almost all interval
exchange transformations for some special irreducible permutations
are disjoint from ELF automorphisms. More precisely, if $\pi$ is
an irreducible permutation of $m$ elements such that
$\pi(i)+1\neq\pi(i+1)$ for $1\leq i\leq m-1$ and
\begin{equation}\label{waredu}
\pi(\pi^{-1}(1)-1)=\pi(1)-1\ \text{ or }\
\pi(\pi^{-1}(m)+1)=\pi(m)+1
\end{equation} then almost every
corresponding interval exchange transformation is disjoint from
all ELF systems. This class of permutation is in a sense marginal
because the corresponding IETs can be "reduced" to exchanges of
$m-1$ intervals. In this paper we will prove the disjointness of
almost every IET from ELF automorphisms (see
Theorem~\ref{koncowe}) for more interesting set of permutations
lying in the Rauzy class generated by the symmetric permutation
$$\tau_m^{sym}=\begin{pmatrix}
    1&2&\ldots&m-1&m\\
    m&m-1&\ldots&2&1\\
        \end{pmatrix},$$
where $m\geq 5$ is odd. As $\tau_3^{sym}$ fulfills (\ref{waredu}),
the disjointness for $m=3$ was proved in \cite{De-Fr-Le-Pa}.

The idea of the proof consists in finding some IETs  which have
segments of orbits with special code properties. For $m=5$ we deal
with IETs of periodic type. Since such IETs are isomorphic to
substitutional dynamical systems, we can use an approach developed
in \cite{Br-Fr} for substitutions for which one is an eigenvalue
of the associated matrix. It allows us to find the required
segments of orbits. In order to find appropriate IETs of periodic
type we use a method of searching some paths in Rauzy graphs
developed in \cite{Ma-Mo-Yo} and \cite{Si-Ul}. Moreover, we
introduce a procedure which helps us to reduce the problem of
finding appropriate IETs  to a smaller number of intervals. Since
the existence of the required segments of orbits is an open
condition, we can use standard technics introduced by Veech in
\cite{Ve4} to prove the existence of such segments for almost all
IETs in the Rauzy class of the permutation $\tau_m^{sym}$.
Finally, this allows us to apply Proposition~\ref{proprozaut}.

As a byproduct, we obtain (using Maple) concrete examples of IETs
of periodic type disjoint from ELF systems.

Using Proposition~\ref{proprozlflow}, for some irreducible
permutations we show that in the class of special flows built over
corresponding IETs and under roof functions constant over the
exchanged intervals almost every such flow is disjoint from weakly
mixing ELF flows (see Corollary~\ref{corflow}).

\section{Preliminaries}
\subsection{Joinings}
For background information on
the theory of joinings  we refer the reader to \cite{Gl} and
\cite{Rue}.

Let $T$ and $S$ be ergodic automorphisms of standard probability
Borel spaces $(X,\mathcal{B},\mu)$ and $(Y,\mathcal{C},\nu)$
respectively. By a \emph{joining} between $T$ and $S$ we mean any
$T\times S$-invariant probability measure $\rho$ on $(X\times Y,
\mathcal{B}\otimes\mathcal{C})$ whose projections on $X$ and $Y$
are equal to $\mu$ and $\nu$ respectively. If we consider flows
$\mathcal{T}=\{T_t\}_{t\in\mathbb{R}}$ and
$\mathcal{S}=\{S_t\}_{t\in\mathbb{R}}$, by a joining between
$\mathcal{T}$ and $\mathcal{S}$ we mean any probability
$\{T_t\times S_t\}_{t\in\mathbb{R}}$--invariant measure on
$(X\times Y,\mathcal{B}\otimes\mathcal{C})$ with the same
property. The set of joinings between automorphisms $T$ and $S$ is
denoted by $J(T,S)$. If the automorphism $T\times S$ on $(X\times
Y,\mathcal{B}\otimes\mathcal{C},\rho)$ is ergodic then the joining
$\rho$ is called ergodic. The set of ergodic joinings is denoted
by $J^e(T,S)$.

Any automorphism $R:(X,\mathcal{B},\mu)\to(X,\mathcal{B},\mu)$
determines a unitary operator on $L^2(X,\mathcal{B},\mu)$, still
denoted by $R$, by the formula $R(f)=f\circ R$. An operator
$\Phi:L^2(X,\mathcal{B},\mu)\rightarrow L^2(Y,\mathcal{C},\nu)$ is
called a \emph{Markov operator} if
$$\Phi(1)=\Phi^{*}(1)=1\text{ and }\Phi(f)\geq 0\text{ for all }f\geq 0.$$
Denote by $\mathcal{J}(S,T)$ the set of all Markov operators from
$L^2(X,\mathcal{B},\mu)$ to $L^2(Y,\mathcal{C},\nu)$ such that $
T\circ\Phi=\Phi\circ  S$. For $\rho\in J(S,T)$, define an operator
$\Phi_{\rho}:L^2(X,\mathcal{B},\mu)\rightarrow
L^2(Y,\mathcal{C},\nu)$ by the formula
\begin{equation} \label{markov}
\langle f,g\rangle_{L^{2}(X\times Y,\rho)}= \langle
\Phi_{\rho}(f),g\rangle_{L^{2}(Y,\nu)}
\end{equation}
for each $f\in L^2(X,\mathcal{B},\mu)$ and $g\in
L^2(Y,\mathcal{C},\nu)$. Then $\phi_{\rho}\in\mathcal{J }(T,S)$
and $J(T,S)\ni\rho\mapsto\phi_{\rho}\in\mathcal{J }(T,S)$
establishes a 1-1 affine correspondence between the sets
$\mathcal{J}(T,S)$ and $J(T,S)$. The set of intertwining Markov
operators $\mathcal{J}(S,T)$ is compact in the weak operator
topology, thus the set $J(S,T)$ is also compact (on $J(T,S)$ we
transport the topology of $\mathcal{J}(T,S)$). Markov operators
corresponding to ergodic joinings are called
\emph{indecomposable}.

Following \cite{Fu}, automorphisms $S$ and $T$ are called
\emph{disjoint} if $J(S,T)=\{\mu\otimes\nu\}$. Note that
$\Phi_{\mu\otimes\nu}f=\int_{X}fd\mu$.

If $S=T$ we will write $J(T)$, $J^e(T)$, $\mathcal{J}(T)$ and
$\mathcal{J}^e(T)$ instead of $J(T,T)$, $J^e(T,T)$,
$\mathcal{J}(T,T)$ and $\mathcal{J}^e(T,T)$ respectively, and the
elements of $J(T)$ are called  \emph{self-joinings} of $T$.

By $C(T)$ we denote the centralizer of $T$, that is the set of all
automorphisms of $(X,\mathcal{B},\mu)$ commuting with $T$. If
$R\in C(T)$ then the formula $\mu_R(A\times B)=\mu(A\cap R^{-1}B)$
determines an ergodic self-joining of $T$ supported on the graph
of $R$. Moreover, $\Phi_{\mu_R}=R$.

\begin{definition}
Let $T$ be an ergodic automorphism of a standard probability space
$(X,\mathcal{B},\mu)$. Following \cite{Fr-Le1}, we say that $T$ is
an \emph{ELF automorphism}, if $\overline{\{ {T}^n:
n\in\mathbb{Z}\}}\subset \mathcal{J}^e(T)$ (the closure is taken
in the weak operator topology). An ergodic flow
$\mathcal{T}=\{T_t\}_{t\in\mathbb{R}}$ on  $(X,\mathcal{B},\mu)$
is said to be an \emph{ELF flow} if $\overline{\{
{T_t}:t\in\mathbb{R}\}}\subset \mathcal{J}^e(\mathcal{T})$.
\end{definition}

\begin{lemma}\label{dojoi}
Let $\mathcal{T}=\{T_t\}_{t\in\mathbb{R}}$ be a measure-preserving
 flow on a standard probability space
$(X,\mathcal{B},\mu)$ and let $\Phi\in\mathcal{J}(\mathcal{T})$.
Suppose that $\{t_n\}$ is a real sequence such that
$t_n\to+\infty$ and the sequence of operators $\{T_{t_n}\}$
converges in the weak operator topology. For every $0<\alpha\leq
1$ the  following two statements are equivalent:
\begin{itemize}
\item[(i)] There exists $\Phi'\in\mathcal{J}(\mathcal{T})$ such that
$T_{t_n}\to\alpha\Phi+(1-\alpha)\Phi'\text{ weakly.}$
\item[(ii)]
For all $A,B\in\mathcal{B}$ we have
$\lim_{n\rightarrow\infty}\langle
T_{t_n}\chi_A,\chi_B\rangle\geq\alpha\langle
\Phi\chi_A,\chi_B\rangle.$
\end{itemize}
\end{lemma}

\begin{proof}
Suppose that  $T_{t_n}\to\alpha\Phi+(1-\alpha)\Phi'$ weakly for
some $\Phi'\in\mathcal{J}(\mathcal{T})$.  Then
\[\langle T_{t_n}\chi_A,\chi_B\rangle\to\alpha\langle
\Phi\chi_A,\chi_B\rangle+(1-\alpha)\langle
\Phi'\chi_A,\chi_B\rangle\] for all $A,B\in\mathcal{B}$. Since
$\Phi'\chi_A\geq 0$, we have $\langle
\Phi'\chi_A,\chi_B\rangle\geq 0$, which implies (ii).

The converse  follows directly from Lemma 3 in \cite{Fr-Le2}.
\end{proof}
\begin{remark}
Note that if  $\rho,\rho_0\in J(\mathcal{T})$ are self-joinings
such that $T_{t_n}\to\Phi_\rho$ weakly and $\Phi_{\rho_0}=\Phi$
then condition (ii) in Lemma~\ref{dojoi} is equivalent to
$\rho\geq\alpha\rho_0$.
\end{remark}
\begin{lemma}\label{alfy}
Let $\mathcal{T}=\{T_t\}_{t\in\mathbb{R}}$ be an ergodic flow on a
standard probability space $(X,\mathcal{B},\mu)$. Suppose that
there exist a real sequence $\{t_n\}$ with $t_n\to+\infty$,
$\theta\neq 0$, $0<\alpha_1,\alpha_2\leq 1$ and
$\Phi_1,\Phi_2\in\mathcal{J}(\mathcal{T})$ such that
\[T_{t_n}\to\alpha_1Id+(1-\alpha_1)\Phi_1\text{ and }T_{t_n}\to\alpha_2 T_\theta+(1-\alpha_2)\Phi_2\]
weakly in $\mathcal{L}(L^2(X,\mathcal{B},\mu))$. Then
$\alpha_1+\alpha_2\leq 1$ and there exists
$\Phi\in\mathcal{J}(\mathcal{T})$ such that
\[T_{t_n}\to\alpha_1Id+\alpha_2
T_\theta+(1-\alpha_1-\alpha_2)\Phi\text{ weakly in
}\mathcal{L}(L^2(X,\mathcal{B},\mu)).\]
\end{lemma}

\begin{proof}
Let $\rho\in J(\mathcal{T})$ be a self-joining such that
$T_{t_n}\to\Phi_\rho$ weakly. By Lemma~\ref{dojoi},
$\rho\geq\alpha_1\mu_{Id}$ and $\rho\geq\alpha_2\mu_{T_{\theta}}$.
By the ergodicity of $\mathcal{T}$, $\mu(\{x\in X:T_\theta
x=x\})=0$, and hence the measures $\mu_{Id}$ and $\mu_{T_\theta}$
are orthogonal. It follows that
$$\rho\geq\alpha_1\mu_{Id}+\alpha_2\mu_{T_{\theta}}=(\alpha_1+\alpha_2)\frac{\alpha_1\mu_{Id}+\alpha_2\mu_{T_{\theta}}}{\alpha_1+\alpha_2}.$$
Since $\rho$ is probabilistic, we obtain $\alpha_1+\alpha_2\leq
1$. Now we apply Lemma~\ref{dojoi} to the operator associated with
the self-joining
$\frac{\alpha_1\mu_{Id}+\alpha_2\mu_{T_{\theta}}}{\alpha_1+\alpha_2}\in
J(\mathcal{T})$. This completes the proof.
\end{proof}

\subsection{Special flows}
Assume that $T$ is an ergodic automorphism of a standard
probability space $(X,\mathcal{B},\mu)$ and let $f\in
L^{1}(X,\mathcal{B},\mu)$ be a positive function. Denote by
$\lambda$ the Lebesgue measure on $\mathbb{R}$. Let
$X^f=\{(x,r)\in X\times\mathbb{R};\ 0\leq r<f(x)\}$ and let
$\mathcal{B}^f$ and $\mu^f$ be the restrictions of $\mathcal{
B}\otimes\mathcal{ B}(\R)$ and $\mu\otimes\lambda$ to $X^f$.
Denote by
$T^f=\{(T^f)_t\}_{t\in\mathbb{R}}:(X^f,\mathcal{B}^f,\mu^f)\rightarrow(X^f,\mathcal{B}^f,\mu^f)$
 the \emph{special flow} built over $T$ and under $f$.
The special flow moves each point in $X^f$ vertically at unit
speed and points $(x,f(x))$ and $(Tx,0)$ are identified. For any
$n\in\mathbb{Z}$ let
\begin{displaymath}
        f^{(n)}(x)=\left\{
        \begin{array}{cll}
        f(x)+\ldots+f(T^{n-1}x)& \hbox{ if } &   n>0, \\
        0  & \hbox{ if } &  n=0, \\
        -(f(T^nx)+\ldots+f(T^{-1}x))& \hbox{ if } &   n<0.
        \end{array}
        \right.
        \end{displaymath}
Consider the skew product
$S_{-f}:(X\times\mathbb{R},\mathcal{B}\otimes\mathcal{B}(\mathbb{R}),\mu\otimes\lambda)\rightarrow(X\times\mathbb{R},\mathcal{B}\otimes\mathcal{B}(\mathbb{R}),\mu\otimes\lambda)$
given by $S_{-f}(x,r)=(Tx,-f(x)+r).$ Thus
$(S_{-f})^k(x,r)=(T^kx,-f^{(k)}(x)+r)$ for each $k\in\mathbb{Z}.$
Moreover,
$$T^f_t(x,r)=(S_{-f})^k(x,r+t)=(T^kx,r+t-f^{(k)}(x)),$$
where $k\in \Z$ is given by $f^{(k)}(x)\leq r+t<f^{(k+1)}(x)$

If $f\equiv 1$, the special flow $T^f$ is called a
\emph{suspension flow}. For more information on special flows we
refer the reader to \cite{Co-F0-Si}.

To apply Propositions~\ref{proprozlflow} and \ref{proprozaut} we
will need the following result from \cite{Fr-Le2}.

\begin{proposition}\label{twozbdoc}
Let $(X,d)$ be a compact metric space. Let $\mathcal{B}$ stand for
the $\sigma$--algebra of  Borel subsets of $X$ and let $\mu$ be a
probability Borel measure on $X$. Suppose that $T:\xbm\to\xbm$ is
an ergodic measure--preserving automorphism and there exist an
increasing sequence of natural numbers $\{q_n\}$ and a sequence of
Borel sets $\{C_n\}$ such that
\[\mu(C_n)\to\alpha>0,\;\;\mu(C_n\triangle T^{-1}C_n)\to 0\;\;\mbox{ and }
\sup_{x\in C_n}d(x,T^{q_n}x)\to 0.\] Assume that $f\in L^2(X,\mu)$
a positive function bounded away from zero and $\{a_n\}$ is a
sequence of real numbers such that the sequence
$\left\{\int_{C_n}|f^{(q_n)}(x)-a_n|^2d\mu(x)\right\}$ is bounded.
Suppose that
\[\frac{1}{\mu(C_n)}\left((f^{(q_n)}-a_n)|_{C_n}\right)_*(\mu|_{C_n})\to P\]
weakly in  the set of probability Borel measures on $\R$.
 Then $\{(T^f)_{a_n}\}$ converges weakly to the operator
\[\alpha\int_{\mathbb{R}}(T^f)_{-t}\,dP(t)+(1-\alpha)J,\]
where $J\in \mathcal{J}(T^f)$.
\end{proposition}

\subsection{Substitutions}\label{rekurencyjne}
Let us consider a finite alphabet $A=\{0,\ldots,r-1\}$  and let
$A^{*}$ stand for the set of nonempty finite words over $A$. Each
map $\sigma:A\rightarrow A^{*}$ is called a \emph{substitution}.
The map $\sigma$ can be extended to $\sigma:A^{*}\rightarrow
A^{*}$, $\sigma:A^{\mathbb{N}}\rightarrow A^{\mathbb{N}}$ and
$\sigma:A^{\mathbb{Z}}\rightarrow A^{\mathbb{Z}}$ by taking
concatenations. The \emph{associated matrix} with the substitution
$\sigma$ is the matrix $M=[m_{ij}]_{0\leq i,j\leq r-1}$, where
$m_{i,j}$ is the number of occurrences of the symbol $i$ in the
word $\sigma(j)$. A substitution $\sigma$ is called
\emph{primitive}, if there exists $n\in\mathbb{N}$ such that all
entries of $M^{n}$ are strictly positive. Perron-Frobenius theorem
states that for each primitive matrix, there exists an eigenavalue
$\theta>0$, which is greater than the absolute value of any other
eigenvalue of $M$. Moreover, there exist left and  right
eigenvectors associated with $\theta$ with positive entries.

The space $A^{\mathbb{Z}}$ is endowed with the metric
$$d(x,y)=\frac{1}{1+\inf\{|k|; x_k\neq y_k\}}$$
for each $x,y\in A^{\mathbb{Z}}$. For any primitive substitution
$\sigma$ there is at least one sequence $u\in A^{\Z}$ such that
$u=\sigma^k(u)$ for some $k\geq 1$. If the sequence $u$ is not
periodic, the substitution $\sigma$ is called \emph{aperiodic}.
Let $S:A^{\Z}\rightarrow A^{\Z}$ stand for the left shift defined
by $(Sx)_n=x_{n+1}$. Let $X_{\sigma}=\overline{\{S^{n}u:\
n\in\mathbb{Z}\}}$. Denote by $\mathcal{L}_{\sigma}\subset A^*$
the language which contains of all finite words which occur in the
sequence $u$. Therefore $\{x_n\}_{n\in\Z}\in X_{\sigma}$ if and
only if $x_n\ldots x_{n+k}\in\mathcal{L}_{\sigma}$ for all
$n\in\Z$ and $k\geq 0$. Since $X_{\sigma}$ is $S$-invariant, we
can consider the restriction of $S$ to $X_{\sigma}$. This
topological dynamical system will be denoted by
$S_{\sigma}:X_{\sigma}\to X_{\sigma}$ and is called a
\emph{substitution dynamical system}.
 If $\sigma$
is primitive and aperiodic then the dynamical system
$(X_{\sigma},S_\sigma)$ is minimal and uniquely ergodic (see
\cite{Qu}). Denote by $\mu_\sigma$ the unique $S_\sigma$-invariant
probability measure.

For any  $\bar{w}=w_0\ldots w_{m-1}\in A^*$, the vector
$l(\bar{w})=(l_0(\bar{w}),\ldots,l_{r-1}(\bar{w}))$ with
$l_i(\bar{w})=\#\{j:w_j=i,0\leq j\leq m-1\}$ for any $0\leq i\leq
r-1$, is called the \emph{population vector} of the word
$\bar{w}$. Following \cite{Cl-Sa}, a finite word
$\bar{w}=w_0\ldots w_{k-1}\in \mathcal{L}_{\sigma}$ is called a
\emph{recurrence word} of $\sigma$ if $\bar{w}w_0=w_0\ldots
w_{k-1}w_0\in \mathcal{L}_{\sigma}$.

\section{Interval exchange transformations}
\subsection{Introduction}
Let $m\geq 2$ be a fixed natural number and let
$$\R_+^m=\{\lambda=(\lambda_1,\ldots,\lambda_m)\in\mathbb{R}^m\setminus\{0\}:\ \lambda_i\geq 0,\ 1\leq i\leq m\}.$$
Denote by $S_m$ the set of all permutations of $\{1,\ldots,m\}$. A
permutation $\pi$ is called \emph{irreducible}, if
$\pi\{1,\ldots,k\}\neq\{1,\ldots,k\}$ for any $1\leq k\leq m-1$.
The set of all irreducible permutations is denoted by $S_m^0$.
Given $(\lambda,\pi)\in\R_+^m\times S^0_m$ let
$T=T_{\lambda,\pi}:[0,|\lambda|)\rightarrow[0,|\lambda|)$
($|\lambda|=\sum_{i=1}^m\lambda_i$) stand for  the \emph{interval
exchange transformation (IET)} of $m$ intervals
$\Delta_j=\left[\sum_{i=1}^{j-1}\lambda_{i},\sum_{i=1}^j\lambda_i\right)$,
$j=1,\ldots,m$, which are rearranged according to the permutation
$\pi$. Denote by $\mu$ the restriction of the Lebesgue measure on
$\R$ to the interval $[0,|\lambda|)$. Then
$T:([0,|\lambda|),\mu)\rightarrow([0,|\lambda|),\mu)$ is a
measure--preserving automorphism. Note that $T^{-1}$ is also an
IET of intervals
$T\Delta_{\pi^{-1}(1)},\ldots,T\Delta_{\pi^{-1}(m)}$.

Denote by $\beta_0,\ldots,\beta_{m-1}$  the left endpoints of the
intervals $\Delta_1,\ldots,\Delta_m$. Let $Orb(x)=\{T^nx:\
n\in\mathbb{Z}\}$ stand for the orbit of point $x$. Following
\cite{Ke}, we say that $T$ satisfies the \emph{infinite distinct
orbit condition (IDOC)}, if every orbit of $\beta_s$, $1\leq s\leq
m-1$, is infinite and $Orb(\beta_s)\cap Orb(\beta_t)=\emptyset$
for all $1\leq s\neq t\leq m-1$. Recall that (see \cite{Ke}) if
$\pi\in S_m^0$ then $T_{\lambda,\pi}$ satisfies the IDOC for a.e.
$\lambda\in\mathbb{R}_{+}^{m}$. Moreover, every IET satisfying the
IDOC is minimal, i.e.\ every orbit is dense (see \cite{Ke}).

For any subinterval $Z\subset[0,|\lambda|)$ we can define the
induced IET $T_Z:Z\rightarrow Z$ by the formula
$T_Z(x)=T^{k(x)}(x)$ for any $x\in Z$, where $k(x)\in\N$ is the
first positive return time of $x$ to $Z$. The induced map $T_Z$ is
an exchange of at most $m+2$ intervals (see e.g.\
\cite{Co-F0-Si}).

Let $d(\lambda,\lambda')=\sum_{i=1}^{m}|\lambda_i-\lambda_i'|$ for
all $\lambda,\lambda'\in\mathbb{R}^m_{+}$. We will consider on
$\R^m_+\times S_m^0$ the metric
$$\bar{d}((\lambda,\pi),(\lambda',\pi'))=d(\lambda,\lambda')+\delta_{\pi,\pi'}.$$

\subsection{Rauzy induction} Suppose that
$T=T_{\lambda,\pi}:[0,|\lambda|)\rightarrow[0,|\lambda|)$ is an
IET given by $\lambda=(\lambda_1,\ldots,\lambda_m)\in\R_+^m$ and
permutation $\pi\in S_m^0$. Assume that $T_{\lambda,\pi}$
satisfies the IDOC. By \emph{Rauzy induction} of IETs we mean a
special kind of induction, where
$Z:=[0,|\lambda|-\min\{\lambda_m,\lambda_{\pi^{-1}(m)}\})$. By
IDOC, $\lambda_m\neq\lambda_{\pi^{-1}(m)}$. Then $T_Z$ is again an
exchange of $m$ intervals and we can write
$T_Z=T_{\lambda',\pi'}$, where $(\lambda',\pi')\in\R^m_+\times
S_m^0$. Let us consider two  maps $a,b:S_m^0\rightarrow S_m^0$
\[ a\pi(i)=\left\{
        \begin{array}{ll}
        \pi(i), & i\leq\pi^{-1}(m) \\
        \pi(m), & i=\pi^{-1}(m)+1, \\
        \pi(i-1), & i>\pi^{-1}(m)+1 \\
        \end{array}
        \right.
        b\pi(i)=\left\{
        \begin{array}{ll}
        \pi(i), & \pi(i)\leq\pi(m) \\
        \pi(i)+1, & \pi(m)<\pi(i)<m \\
        \pi(m)+1, & \pi(i)=m. \\
        \end{array}
        \right.\]
and two $m\times m$ matrices
$$A(a,\pi)=
\begin{pmatrix}
1 & 0 & \ldots & 0  & 0 &0 &\ldots&0\\
0 & 1 & \ldots & 0  & 0 &0 &\ldots&0\\
\vdots & \vdots & \ddots & \vdots & \vdots & \vdots& \ddots&\vdots\\
0 & 0 & \ldots & 1  & 1 &0 &\ldots&0\\
0 & 0 & \ldots & 0  & 0 &1 &\ldots&0\\
\vdots & \vdots & \ddots & \vdots & \vdots & \vdots&\ddots &\vdots\\
0 & 0 & \ldots & 0  & 0 &0 &\ldots&1\\
0 & 0 & \ldots & 0  & 1 &0 &\ldots&0\\
\end{pmatrix}
\begin{matrix}
\\ \\ \leftarrow\pi^{-1}(m)\\ \\ \\ \\
\end{matrix}
$$
$$A(b,\pi)=\begin{matrix}\begin{matrix}&&&&&\end{matrix}\\\begin{pmatrix}
1&&\ldots&&0\\
\vdots&&\ddots&&\vdots\\
0&\ldots&0\,1\,0&\ldots&1\\
\end{pmatrix}.\\
\begin{matrix}&&\stackrel{\uparrow}{\pi^{-1}(m)}&&\end{matrix}
\end{matrix}$$
The induced IET is then given by
$(\lambda',\pi')=(A(c,\pi)^{-1}\lambda,c(\pi))$, where
\[c=c(\lambda,\pi)=\left\{\begin{matrix}a&\text{ if }&\lambda_m<\lambda_{\pi^{-1}(m)}\\
b&\text{ if
}&\lambda_m>\lambda_{\pi^{-1}(m)}.\end{matrix}\right.\] By IDOC,
we can define Rauzy induction for the IET given by
$(\lambda',\pi')$. In fact, this procedure can be continued for
any finite number of steps.  Let
$\Lambda_m=\{\lambda\in\R_+^{m}:|\lambda|=1\}$. Define maps
$$J:\R_+^{m}\times S_m^0\rightarrow\R_+^{m}\times S_m^0,\ J(\lambda,\pi)=(A^{-1}(c,\pi)\lambda,c\pi),\ c=c(\lambda,\pi)$$
and
$$P:\Lambda_m\times S_m^0\rightarrow\Lambda_m\times S_m^0,\ P(\lambda,\pi)=\left(\frac{A^{-1}(c,\pi)\lambda}{|A^{-1}(c,\pi)\lambda|},c\pi\right),\
c=c(\lambda,\pi).$$ If $(\lambda',\pi'):=J^n(\lambda,\pi)$ then
$$\pi'=\pi_n,\ \lambda'=A(c_n,\pi_{n-1})^{-1}A(c_{n-1},\pi_{n-2})^{-1}\ldots A(c_2,\pi_1)^{-1}A(c_1,\pi)^{-1}\lambda,$$
where $$c_k=c_k(\lambda,\pi)=c(J^{k-1}(\lambda,\pi)),\
\pi_k=c_k\circ\ldots\circ c_1\circ\pi.$$ Hence
$$\lambda=A^{(n)}(\lambda,\pi)\cdot\lambda',\text{ where }
A^{(n)}(\lambda,\pi)=A(c_1,\pi)\ldots A(c_n,\pi_{n-1}).$$

By \emph{Rauzy graph} we mean a directed graph  whose vertices are
irreducible permutations and edges connect permutations obtained
one from the other by applying maps $a$ or $b$ and are labeled
according to the type, $a$ or $b$ respectively. Any connected
component of the Rauzy graph is called a \emph{Rauzy class}.
Denote by $\mathcal{R}(\pi)\subset S_m^0$ the Rauzy class
containing the permutation $\pi\in S_m^0$. See \cite{Ra},
\cite{Ve3}, \cite{Ve4} for more details.

\begin{theorem}[\cite{Ve3}]\label{tw_veech_o_mierze_na_klasie} Let $\mathcal{R}\subset S_m^0$
be a fixed Rauzy class. There exists on $\Lambda_m\times
\mathcal{R}$ a smooth positive $\sigma$-finite $P$-invariant
measure $\kappa$, which is ergodic and conservative with respect
to $P$ and which is equivalent to the Lebesgue measure.
\end{theorem}

\begin{theorem}[\cite{Ma}, \cite{Ve3}]\label{tw_veech_o_mierze}
If $\pi\in S_m^0$ then for a.e. $\lambda\in\R_+^m$ the IET
$T_{\lambda,\pi}$ is uniquely ergodic.
\end{theorem}

\begin{notation}
Let $T_{\lambda,\pi}:[0,|\lambda|)\to[0,|\lambda|)$ be an IET.
Each point $x\in[0,|\lambda|)$ is coded by a sequence
$\{w_n\}_{n\in\Z}\in\{1,\ldots,m\}^\Z$ so that
$T^nx\in\Delta_{w_n}$ for $n\in\Z$. Denote by
$\mathcal{L}(T_{\lambda,\pi})\subset\{1,\ldots,m\}^*$ the language
determined by all such sequences, i.e.\
$\mathcal{L}(T_{\lambda,\pi})$ is the set of all finite words
(over the alphabet $\{1,\ldots,m\}$) which occur in such
sequences. Therefore $\bar{w}=w_0\ldots
w_{k-1}\in\mathcal{L}(T_{\lambda,\pi})$ if and only if  the set
$\bigcap_{j=0}^{k-1}T_{\lambda,\pi}^{-j}\Delta_{w_j}$ is not
empty. A word $\bar{w}=w_0\ldots
w_{k-1}\in\mathcal{L}(T_{\lambda,\pi})$ is called
\emph{recurrence} if $w_0\ldots
w_{k-1}w_0\in\mathcal{L}(T_{\lambda,\pi})$. As in
Subsection~\ref{rekurencyjne} for every $\bar{w}=w_0\ldots
w_{k-1}\in\mathcal{L}(T_{\lambda,\pi})$ denote by
$l(\bar{w})=(l_1(\bar{w}),\ldots,l_{m}(\bar{w}))^T$ the
\emph{population vector}, i.e.\ $l_i(\bar{w})=\#\{j:w_j=i,0\leq
j\leq k-1\}$ for $1\leq i\leq m$.
\end{notation}

\begin{remark}\label{uwagaoh}
Suppose that $T=T_{\lambda,\pi}:[0,|\lambda|)\to[0,|\lambda|)$
fulfills the IDOC. Consider $T_{(n)}=T_{\lambda^{(n)},\pi^{(n)}}$,
where $(\lambda^{(n)},\pi^{(n)})=J^n(\lambda,\pi)$. Let
$\Delta^{(n)}_1,\ldots,\Delta^{(n)}_m$ stand for the intervals
exchanged by $T_{(n)}$. Then the action of the initial IET $T$ can
be seen in terms of Rohlin towers over
$\Delta^{(n)}_1,\ldots,\Delta^{(n)}_m\subset[0,|\lambda|)$. Let
\[h^{(n)}=(h^{(n)}_1,\ldots,h^{(n)}_m)=(1,\ldots,1)A^{(n)}(\lambda,\pi).\]
Then $h^{(n)}_i$ is the first return time for the action of $T$ on
$\Delta^{(n)}_i$ to $[0,|\lambda^{(n)}|)$ for $1\leq i\leq m$ and
\[A^{(n)}(\lambda,\pi)_{ij}=\#\{0\leq k<h^{(n)}_i:T^k\Delta^{(n)}_j\subset \Delta_i\}.\]
 Moreover, $\Xi_i=\{T^j\Delta^{(n)}_i:0\leq j<h^{(n)}_j\}$ is a
Rohlin tower of intervals and the towers $\Xi_1,\ldots,\Xi_m$ are
pairwise disjoint and fill the whole interval $[0,|\lambda|)$. It
follows that
\begin{equation}\label{wypelnic}
\sum_{i=1}^m h^{(n)}_j\lambda^{(n)}_j=|\lambda|.
\end{equation}
\end{remark}
\begin{remark}\label{iter}
For any $\bar{w}=w_0\ldots w_{K-1}\in\mathcal{L}(T_{(n)})$ let
$q^{(n)}_{\bar{w}}=\sum_{i=0}^{K-1}h^{(n)}_{w_i}$. If
$x\in[0,|\lambda^{(n)}|)$ and $T_{(n)}^ix\in\Delta^{(n)}_{w_i}$
for $0\leq i< K$ then $T_{(n)}^Kx=T^{q^{(n)}_{\bar{w}}}x$.
Moreover,
\begin{equation}
q^{(n)}_{\bar{w}}=\sum_{i=0}^{K-1}h^{(n)}_{w_i}=h^{(n)}l(w_0\ldots
w_{K-1})=|A^{(n)}(\lambda,\pi)l(w_0\ldots w_{K-1})|,
\end{equation}
where $l(w_0\ldots w_{K-1})$ is the population vector of the word
$w_0\ldots w_{K-1}$. Next consider the word $\bar{w}'=w_0'\ldots
w_{q^{(n)}_{\bar{w}}-1}'\in\mathcal{L}(T)$ such
$T^jx\in\Delta_{w_j'}$ for $0\leq j<q^{(n)}_{\bar{w}}$. Thus
\begin{equation}\label{zmianah}
l(\bar{w}')=A^{(n)}(\lambda,\pi)l(\bar{w}).
\end{equation}
\end{remark}

\begin{definition}[see \cite{Si-Ul}] An
IET $T_{\lambda,\pi}$ is called of \emph{periodic type}, if the
following holds:
\begin{itemize}
\item[(i)] There exists $k>1$ such that
$P^{n+k}(\lambda,\pi)=P^{n}(\lambda,\pi)$ for all
$n\in\mathbb{N}$;
\item[(ii)] The matrix $A^{(k)}(\lambda,\pi)$ has strictly positive entries.
\end{itemize}
\end{definition}
The matrix $(A^{(k)}(\lambda,\pi))^T$ is also known as a
\emph{period matrix}.
It was shown in \cite{Ma-Mo-Yo} how to produce matrices with
strictly positive entries by walking on the Rauzy graph.
Furthermore, in \cite{Si-Ul} (see Lemma 6) the authors presented a
simple method for searching IETs of periodic type using closed
paths in the Rauzy graph.
\begin{remark}\label{periodic-idoc}
Each IET of periodic type satisfies the IDOC (see \cite{Si-Ul}).
\end{remark}

\subsection{Permutation $\eta_{\pi}$ and subspace $H(\pi)$}
\label{permutacja_eta}
Let $\pi\in S_m^0$.
Following \cite{Ve3}, define the permutation $\eta_{\pi}$ on
$\{0,\ldots,m\}$ as follows:
\begin{displaymath}
        \eta_{\pi}(i)=\left\{
        \begin{array}{cl}
        \pi^{-1}(1)-1 &\text{ if } i=0 \\
        m & \text{ if } i=\pi^{-1}(m) \\
        \pi^{-1}(\pi(i)+1)-1 &  \mbox{ otherwise}. \\
        \end{array}
        \right.
        \end{displaymath}
The set $\{0,\ldots,m\}$ is partitioned by $\eta_{\pi}$ into
\emph{cyclic subsets}, i.e.\ into the orbits of $\eta_{\pi}$.
Denote by $\Sigma(\pi)$ the set of all cyclic sets of the
permutation $\eta_{\pi}$. To each $S\in\Sigma(\pi)$ we associate a
vector $b(S)\in\mathbb{Z}^m$ given by
$$b(S)_i=\chi_S(i-1)-\chi_S(i),\ 1\leq i\leq m,$$
where $\chi_S$ is the characteristic function of $S$. By $|b(S)|$
denote the sum of entries of the vector $b(S)$.
\begin{proposition}[\cite{Ve4}, Proposition 5.11]\label{aaa} For each
$S\in\Sigma(\pi)$,
\begin{displaymath}
        |b(S)|=\left\{
        \begin{array}{rl}
        1 & \text{ if }\  0\in S,\ m\notin S, \\
        -1 &\text{ if }\ 0\notin S,\ m\in S, \\
        0 & \mbox{ otherwise}.
        \end{array}
        \right.
        \end{displaymath}
\end{proposition}

\begin{lemma}[\cite{Ve4}, Lemma 5.6] \label{ww} Let $\pi\in S^0_m$ and $c=a$ or $b$. There is a bijection
$\Sigma(\pi)\ni S\mapsto cS\in\Sigma(c\pi)$ such that
$$b(S)=A(c,\pi)b(cS)\text{ for every }S\in\Sigma(\pi).$$
\end{lemma}
Let us recall the definition of the alternating $m\times m$ matrix
$L^{\pi}$:
\begin{displaymath}
        L^{\pi}_{ij}=\left\{
        \begin{array}{rl}
        1 & \text{ if }i<j \text{ and }\pi(i)>\pi(j), \\
        -1 & \text{ if }i>j\text{ and }\pi(i)<\pi(j), \\
        0 & \mbox{ otherwise}.
        \end{array}
        \right.
        \end{displaymath}
Define $H(\pi)=L^{\pi}(\mathbb{R}^m)$. Since $L^\pi$ is
anti-symmetric, $H(\pi)=(\operatorname{ker} L^{\pi})^\perp$.

\begin{proposition}[\cite{Ve4}, Proposition 5.2]
\label{nnn} A vector $h\in H(\pi)$ if and only if $h\cdot b(S)=0$
for every $S\in\Sigma(\pi)$.
\end{proposition}

\begin{remark}\label{jjj} Let $\widetilde{S}_m^0$ stand for the subset of irreducible
permutations $\pi$ of $m$ elements such that if $S\in\Sigma(\pi)$,
then $|b(S)|=\pm 1$. The condition $\pi\in\widetilde{S}_m^0$ is
quite restrictive. In view of Proposition~\ref{aaa}, this implies
$\#\Sigma(\pi)=2$, hence if $\Sigma(\pi)=\{S_0,S_1\}$ with $0\in
S_0$  then $m\in S_1$ and $b(S_1)=-b(S_0)$. Note that if $m$ is
odd then $\tau_m^{sym}\in\widetilde{S}_m^0$  and
$b(S_0)=(1,-1,1,\ldots,-1,1)$.
\end{remark}

Let $E=[E_{ij}]_{1\leq i,j\leq m}$ be a matrix with strictly
positive entries. Following \cite{Ve2} set
$$\nu(E)=\max_{1\leq i,j,k\leq m}\frac{E_{ij}}{E_{ik}}\text{ and }e_j=\sum_{i=1}^{m}E_{ij}.$$
Then
\begin{equation}e_j\leq\nu(E)e_k,\ 1\leq j,k\leq m\text{ and }\nu(FE)\leq\nu(E),\label{1}\end{equation}
for any nonnegative nonsingular matrix $F=[F_{ij}]_{1\leq i,j\leq
m}$.

\subsection{IETs and substitutions}\label{kodowanie} For any IET
$T:=T_{\lambda,\pi}:[0,|\lambda|)\rightarrow[0,|\lambda|)$ of $m$
intervals $\Delta_1,\ldots,\Delta_m$ and $n\in\N$ we can define a
substitution on $m$ symbols (see \cite{Fi-Ho-Ro}). We  shortly
describe this procedure. Let $(\lambda',\pi')=J^n(\lambda,\pi)$.
Let $\Delta_1',\ldots,\Delta_m'$ be the exchanged intervals for
$T_{\lambda',\pi'}$. Let $k:[0,|\lambda'|)\rightarrow\mathbb{N}$
stand for the first return time map to the interval
$[0,|\lambda'|$). Let $A=\{1,\ldots,m\}$. We define a map
$cod:[0,|\lambda'|)\rightarrow A^{*}$ in the following way:
$cod(x)=i_1\ldots i_{k(x)}$  if $T^jx\in\Delta_{i_j}$ for $0\leq
j<k(x)$. Since the map $cod:[0,|\lambda'|)\rightarrow A^{*}$ is
constant on every interval $\Delta_1',\ldots,\Delta_m'$, we can
define a substitution $\sigma:A\rightarrow A^{*}$ so that
$\sigma(i)=cod(x)$ for each $x\in\Delta_i'$ and $i=1,\ldots,m$.

\begin{remark}\label{isosub}
Suppose that $T_{\lambda,\pi}$ is an IET of periodic type for
which $A^{(n)}(\lambda,\pi)^T$ is its periodic matrix. Let us
consider the corresponding substitution $\sigma:A\rightarrow A
^{*}$. Then $T_{\lambda,\pi}$ is measure-theoretically isomorphic
to the substitution system $S_\sigma$. The isomorphism is
established by the map
$\phi:[0,|\lambda|)\rightarrow{A}^{\mathbb{Z}}$,
$$(\phi(x))_i=j,\ \mbox{if}\ T^ix\in\Delta_j,\ 1\leq j\leq m,\
\mbox{for all}\ i\in\mathbb{Z}.$$ It follows that
$\mathcal{L}(T_{\lambda,\pi})=\mathcal{L}_{\sigma}$. Moreover,
$A^{(n)}(\lambda,\pi)$ is the substitution matrix of $\sigma$.
\end{remark}

\section{Disjointness from ELF system}
\subsection{Perturbation of IET} \label{zaburzenia}
In this section we state three lemmas which, roughly speaking, say
that  any type of finite orbit combinatorics of any  IET
satisfying the IDOC is preserved in the passage to its slight
perturbation. The proof of the lemmas is rather straightforward
and we leave it to the reader.

Let $T:=T_{\lambda,\pi}:[0,|\lambda|)\rightarrow [0,|\lambda|)$ be
an IET of $m$ intervals $\Delta_1,\ldots,\Delta_m$ of lengths
$\lambda_1,\ldots,\lambda_m\geq 0$ and given by an irreducible
permutation $\pi$.  Recall that
$$Tx=x+\sum_{\pi(j)<\pi(i)}\lambda_{\pi(j)}-\sum_{j<i}\lambda_j\text{ if } x\in\Delta_i,$$
$$T^{-1}x=x-\sum_{\pi(j)<\pi(i)}\lambda_{\pi(j)}+\sum_{j<i}\lambda_j\text{ if } x\in T\Delta_i,$$
and $\beta_0=0$, $\beta_s=\sum_{j=1}^{s}\lambda_j$ for $1\leq
s\leq m-1$. Assume that $T$ satisfies the \emph{weak IDOC}, i.e.\
every orbit of $\beta_s$, $1\leq s\leq m-1$ is infinite and
$Orb(\beta_s)\cap Orb(\beta_t)\neq\emptyset$ for $1\leq s, t\leq
m-1$ implies $\beta_s=\beta_t$.

Fix $K\geq 1$.   Let
$$
0<\delta<\min\{|T^{-j_1}\beta_{t_1}-T^{-j_2}\beta_{t_2}|:
|j_1|,|j_2|\leq K,\,1\leq t_1,t_2\leq
m-1,\beta_{t_1}\neq\beta_{t_2}\}.
$$ Fix an $\varepsilon=\vep_{T,K}>0$ such that
\begin{equation}\label{zalozenie}
m(2K+4)\varepsilon<\frac{1}{10}\delta.
\end{equation} Let
$$\mathcal{K}:=\{(\lambda^{\varepsilon},\pi)\in\R_+^m\times\{\pi\}:\ |\lambda^{\varepsilon}|=|\lambda|,\ |\lambda_i-\lambda_i^{\varepsilon}|<\varepsilon,\ 1\leq i\leq m \}.$$

Suppose that
$T_{\varepsilon}:=T_{\lambda^{\varepsilon},\pi}:[0,|\lambda|)\rightarrow
[0,|\lambda|)$ is an IET such that
$(\lambda^{\varepsilon},\pi)\in\mathcal{K}$. Let
$\Delta^{\varepsilon}_1,\ldots,\Delta^{\varepsilon}_m$ be the
intervals exchanged by $T_{\varepsilon}$ and let
$\beta_0^{\varepsilon}=0$, $\beta_i^{\varepsilon}=\sum_{j\leq
i}\lambda_j^{\varepsilon}$, $i=1,\ldots,m-1$. Note that
\begin{equation*}
|\beta_i-\beta_i^{\varepsilon}|<m\varepsilon\text{ and
}|T\beta_i-T_{\varepsilon}\beta_i^{\varepsilon}|<m\varepsilon\text{
for  }0\leq i\leq m-1.
\end{equation*}
Indeed,
\[|\beta_t-\beta_t^{\varepsilon}|=|\sum_{j\leq t}\lambda_{j}-\sum_{j\leq t}\lambda^{\varepsilon}_{j}|\leq
\sum_{j\leq t}|\lambda_{j}-\lambda^{\varepsilon}_{j}|< m\vep,\]
\[|T\beta_t-T_{\varepsilon}\beta_t^{\varepsilon}|=|\sum_{\pi(j)<\pi(t+1)}\lambda_{j}-\sum_{\pi(j)<\pi(t+1)}\lambda^{\varepsilon}_{j}|\leq
\sum_{\pi(j)<\pi(t+1)}|\lambda_{j}-\lambda^{\varepsilon}_{j}|<
m\vep.\]

\begin{lemma}\label{iteracje_1}
For every $0\leq s\leq K+1$ and $0\leq t\leq m-1$
\begin{equation*}
|T^{-s+1}\beta_t-T_{\varepsilon}^{-s+1}\beta_t^{\varepsilon}|<m(2s+1){\varepsilon}\text{
and } T^{-s}\beta_t\in\Delta_i\text{ implies
}T_\vep^{-s}\beta^\vep_t\in\Delta^{\vep}_i.
\end{equation*}
Moreover,
\[|T^{s}\beta_t-T_{\varepsilon}^{s}\beta_t^{\varepsilon}|<m(2s+1){\varepsilon}\text{
and } T^{s}\beta_t\in\Delta_i\text{ implies
}T_\vep^{s}\beta^\vep_t\in\Delta^{\vep}_i\] for all $0\leq s\leq
K$ and $0\leq t\leq m-1$.\bez
\end{lemma}

\begin{lemma}\label{nowy}
For all $0\leq s_1,s_2\leq K$ and $0\leq t_1,t_2\leq
m-1$,
\[T^{-s_1}\beta_{t_1}< T^{-s_2}\beta_{t_2}\;\Longleftrightarrow\; T^{-s_1}_{\vep}\beta^{\vep}_{t_1}< T_{\vep}^{-s_2}\beta^{\vep}_{t_2}.\]
\bez
\end{lemma}

Take $\bar{w}=w_0\ldots w_K\in\{1,\ldots,m\}^{K+1}$ such that the
set $I_{\bar{w}}=\bigcap_{i=0}^{K}T^{-i}\Delta_{w_{i}}$ is not
empty, i.e. $\bar{w}\in\mathcal{L}(T)$. The set $I_{\bar{w}}$ is
an interval of the form
$[T^{-k_1}\beta_{t_1},T^{-k_2}\beta_{t_2})$, where $0\leq
k_1,k_2\leq K$ and $0\leq t_1,t_2\leq m-1$. It follows that
$|\bigcap_{i=0}^{K}T^{-i}\Delta_{w_{i}}|>\delta$. Let
$$I^\vep_{\bar{w}}=[T_\vep^{-k_1}\beta^\vep_{t_1},T_\vep^{-k_2}\beta^\vep_{t_2}).$$

In view of (\ref{zalozenie}), the following result is a simple
consequence of Lemmas~\ref{iteracje_1} and \ref{nowy}.
\begin{lemma}\label{iteracje_3}
For each $\bar{w}=w_0\ldots w_K\in\{1,\ldots,m\}^{K+1}$ if
$I_{\bar{w}}\neq\emptyset$ and $(\lambda^\vep,\pi)\in\mathcal{K}$
then
$$|I^\vep_{\bar{w}}|\geq \frac{4}{5}|I_{\bar{w}}|\text{ and
}I^\vep_{\bar{w}}\subset\bigcap_{i=0}^{K}T^{-i}_{\varepsilon}\Delta^{\varepsilon}_{w_{i}}.$$
\bez\end{lemma}

\subsection{Disjointness theorems}\label{disjointness}
Fix $\tau\in{S}_m^0$.  Let
$T_{\lambda,\tau}:[0,1)\rightarrow[0,1)$ be an IET with the IDOC,
which has two recurrence words
$\bar{w}_1\in\mathcal{L}(T_{\lambda,\tau})$ and
$\bar{w}_2\in\mathcal{L}(T_{\lambda,\tau})$  with lengths $K_1$,
$K_2$ respectively and  such that
\begin{equation}\label{recwords}
l(\bar{w}_1)-l(\bar{w}_2)=b(S)\text{ for some }S\in\Sigma(\tau).
\end{equation} Let
$K=\max\{K_1,K_2\}+1$ and choose
$\varepsilon=\vep_{T_{\lambda,\tau},K}>0$ such that
(\ref{zalozenie}) holds. Let
$$\mathcal{K}=\{(\lambda^{\varepsilon},\tau):\ |\lambda^{\varepsilon}|=|\lambda|=1,\ |\lambda_i-\lambda_i^{\varepsilon}|<\varepsilon,\ 1\leq i\leq m \}.$$
Choose $\bar{w}_1^{ext}=w^1_0\ldots
w^1_K\in\mathcal{L}(T_{\lambda,\tau})$ and
$\bar{w}_2^{ext}=w^2_0\ldots
w^2_K\in\mathcal{L}(T_{\lambda,\tau})$ which are extensions of
$\bar{w}_1,\bar{w}_2$  such that $w^1_{K_1}=w^1_0$ and
$w^2_{K_2}=w^2_0$. Let
$\theta_r=|\bigcap_{i=0}^KT_{\lambda,\tau}^{-i}\Delta_{w^r_i}^{\lambda,\tau}|$
for $r=1,2$.

Using some standard Veech's arguments (see \cite{Ve4}), there
exist $k>1$ and maps $c_1,\ldots,c_k$ $(c_i=a$ or $b$ for
$i=1,\ldots,k)$ such that if $\tau_i=c_i\circ\ldots \circ
c_1\circ\tau$ $(i=1,\ldots,k)$ then
\begin{itemize}
\item[(i)] $\tau_k=\tau$,
\item[(ii)] the matrix $B=A(c_1,\tau)A(c_2,\tau_1)\ldots A(c_k,\tau_{k-1})$
has strictly positive entries.
\item[(iii)] $c_i=c_i(\rho,\tau)$ for any $1\leq i\leq k$.
\end{itemize}
Let
$\mathcal{M}=\{({B\lambda^{\varepsilon}}/{|B\lambda^{\varepsilon}|},\tau):\
(\lambda^{\varepsilon},\tau)\in\mathcal{K}\}$. Since $\mathcal{K}$
is open, the set $\mathcal{M}$  is also open in
$(\Lambda_m,\bar{d})$, hence
$\mathcal{M}\subset\Lambda_m\times\mathcal{R}(\tau)$ has positive
measure $\kappa$, where $\kappa$ is the invariant measure from
Theorem~\ref{tw_veech_o_mierze_na_klasie}.

\begin{notation}Denote by $\mathcal{A}$ the set of all
$(\rho,\pi)\in\Lambda_m\times \mathcal{R}(\tau)$ such that the IET
$T_{\rho,\pi}$  satisfies the IDOC and there exists an increasing
sequence  $\{k_n\}_{n\in\N}$ of natural numbers such that
$P^{k_n}(\rho,\pi)\in\mathcal{M}$.
\end{notation}

\begin{remark}\label{full}By the ergodicity and conservativity of
$P:(\Lambda_m\times\mathcal{R}(\tau),\kappa)\to(\Lambda_m\times\mathcal{R}(\tau),\kappa)$
(see Theorem~\ref{tw_veech_o_mierze_na_klasie}),
$\kappa((\Lambda_m\times\mathcal{R}(\tau))\setminus\mathcal{A})=0$.
\end{remark}

Take $(\rho,\pi)\in\mathcal{A}$ and let
$T=T_{\rho,\pi}:[0,1)\rightarrow[0,1)$. Let
$\Delta_1,\ldots,\Delta_m$ stand for the intervals exchanged by
$T$. Let $\{k_n\}_{n\in\N}$ be a sequence such that
$P^{k_n}(\rho,\pi)\in\mathcal{M}$. By the definition of
$\mathcal{M}$, $P^{k_n+k}(\rho,\pi)\in\mathcal{K}$ and
$A^{(k_n+k)}(\rho,\pi)=A^{(k_n)}(\rho,\pi)B$.

Let $(\rho^{(k_n)},\tau):=J^{k_n+k}(\rho,\pi)$ and let
$\Delta_{1}^{(k_n)},\ldots,\Delta_{m}^{(k_n)}$ stand for  the
intervals exchanged  by $T_{\rho^{(k_n)},\tau}$. Then
$(\rho^{(k_n)}/|\rho^{(k_n)}|,\tau):=P^{k_n+k}(\rho,\pi)\in\mathcal{K}$.
Let us apply Lemma~\ref{iteracje_3} to
$T_{\rho^{(k_n)}/|\rho^{(k_n)}|,\tau}$ and the words
$\bar{w}^{ext}_1$, $\bar{w}^{ext}_2$. After rescaling  we obtain
two intervals
$$I_1^{(k_n)}\subset \bigcap_{i=0}^{K_1}T^{-i}_{\rho^{(k_n)},\tau}\Delta^{(k_n)}_{w^1_i}\text{ and }I_{2}^{(k_n)}\subset
\bigcap_{i=0}^{K_2}T^{-i}_{\rho^{(k_n)},\tau}\Delta^{(k_n)}_{w^2_i}$$
such that $|I_1^{(k_n)}|\geq \frac{4}{5}\theta_1|\rho^{(k_n)}|$
and $|I_2^{(k_n)}|\geq\frac{4}{5}\theta_2|\rho^{(k_n)}|$.

By Remark~\ref{uwagaoh},
$h_j^{(k_n)}=\sum_{i=1}^{m}A^{(k_n+k)}_{ij}$ is  the first return
time of the interval $\Delta_{j}^{(k_n)}$ to $[0,|\rho^{(k_n)}|)$
for the action of $T$. Let
$$C_{r}^{(k_n)}:=\bigcup_{i=0}^{h_{w^r_0}^{(k_n)}-1}T^iI_r^{(k_n)}\text{ for }r=1,2.$$

\begin{lemma} \label{miara}
There exists $\alpha>0$ such that $\mu(C_{r}^{(k_n)})\geq
\alpha>0$ for any $n\in\N$ and for $r=1,2$.
\end{lemma}
\begin{proof} In view of (\ref{wypelnic}),
$1=\sum_{j=1}^{m}h_j^{(k_n)}\rho_{j}^{(k_n)}$. Moreover, from
$(\ref{1})$,
$$h_j^{(k_n)}\leq\nu(A^{(k_n+k)})h_l^{(k_n)}=\nu\left(A^{(k_n)}(\rho,\pi)B\right)h_l^{(k_n)}\leq\nu(B)h_l^{(k_n)}$$
for all $ 1\leq j,l\leq m$. It follows that
\[1\leq\sum_{j=1}^{m}\nu(B)h_l^{(k_n)}\rho_j^{(k_n)}=\nu(B)h_l^{(k_n)}|\rho^{(k_n)}|\]
for any $1\leq l\leq m$. Since
$|I_r^{(k_n)}|\geq\frac{4}{5}\theta_r|\rho^{(k_n)}|$ for $r=1,2$,
we obtain
$$1\leq
\nu(B)h_{w_0^r}^{(k_n)}\frac{|I_r^{(k_n)}|}{\frac{4}{5}\theta_r}.$$
Hence
$$0<\alpha:=\frac{4\min(\theta_1,\theta_2)}{5\nu(B)}\leq
h_{w_0^r}^{(k_n)}|I_r^{(k_n)}|=\mu\left(\bigcup_{i=0}^{h_{w_0^r}^{(k_n)}-1}T^iI_{r}^{(k_n)}\right)=\mu(C_{r}^{(k_n)})$$
for $r=1,2$.
\end{proof}

Let
$$q_r^{(k_n)}=|A^{(k_n+k)}(\rho,\pi)l(\bar{w}_r)|=\sum_{j=0}^{K_r-1}h^{(k_n)}_{w^r_j}\text{
for }r=1,2.$$ Since $I_r^{(k_n)}\subset
\bigcap_{i=0}^{K_r}T^{-i}_{\rho^{(k_n)},\tau}\Delta^{(k_n)}_{w^r_i}$,
by Remark~\ref{iter},
\begin{equation}\label{doduw}
T_{\rho^{(k_n)},\tau}^{K_r}x=T^{q_r^{(k_n)}}x\text{ for every
}x\in I_r^{(k_n)}.
\end{equation}

\begin{lemma}\label{12} For  $r=1,2$ we have
\begin{itemize}
\item[(i)] $\mu(C_{r}^{(k_n)}\bigtriangleup T^{-1}C_{r}^{(k_n)})\rightarrow 0$ as $n\rightarrow\infty$,
\item[(ii)] $\sup_{x\in C_{r}^{(k_n)}}|x-T^{q_r^{(k_n)}}x|\rightarrow 0$ as $n\rightarrow\infty$.
\end{itemize}
\end{lemma}
\begin{proof}
(i). Since $\{T^iI_r^{(k_n)}:0\leq i<h_{w^r_0}^{(k_n)}\}$ is a
Rohlin tower, $$TC_{r}^{(k_n)}\bigtriangleup C_{r}^{(k_n)}\subset
I_r^{(k_n)}\cup T^{h_{w_0^r}^{(k_n)}}I_r^{(k_n)}.$$ It follows
that
$$\mu(T^{-1}C_{r}^{(k_n)}\bigtriangleup C_{r}^{(k_n)})\leq
2\mu(I_r^{(k_n)})\leq 2|\rho^{(k_n)}|\rightarrow 0\text{ as
}n\rightarrow\infty.$$
 (ii).
For each $x\in C_{r}^{(k_n)}$ there exist $y\in I_r^{(k_n)}$ and
$0\leq i<h_{w^r_0}^{(k_n)}$ such that $x=T^i y$. In view of
(\ref{doduw}) we have
$T^{q_r^{(k_n)}}y=T_{\rho^{(k_n)},\tau}^{K_r}y$, so
\[T^{q_r^{(k_n)}}x=T^{q_r^{(k_n)}}T^i y=T^iT^{q_r^{(k_n)}}y=T^iT_{\rho^{(k_n)},\tau}^{K_r}y.\]
Since $I_r^{(k_n)}\subset
\bigcap_{i=0}^{K_r}T^{-i}_{\rho^{(k_n)},\tau}\Delta^{(k_n)}_{w^r_i}$,
we have $y\in\Delta^{(k_n)}_{w^r_0}$ and
$T_{\rho^{(k_n)},\tau}^{K_r}y\in\Delta^{(k_n)}_{w^r_{K_r}}=\Delta^{(k_n)}_{w^r_0}$.
It follows that $x=T^iy\in T^i\Delta^{(k_n)}_{w^r_0}$ and
$T^{q_r^{(k_n)}}x=T^iT_{\rho^{(k_n)},\tau}^{K_r}y\in
T^i\Delta^{(k_n)}_{w^r_0}$. Since $T^i\Delta^{(k_n)}_{w^r_0}$ is
an interval of length $|\Delta^{(k_n)}_{w^r_0}|$, we obtain
\[|x-T^{q_r^{(k_n)}}x|<|\Delta_{w^r_0}^{(k_n)}|.\]
Consequently, in view of the IDOC,
$$\sup_{x\in
C_{r}^{(k_n)}}|x-T^{q_r^{(k_n)}}x|<|\Delta_{w^r_0}^{(k_n)}|\rightarrow
0\text{ as }n\rightarrow\infty.$$
\end{proof}

Let $b_{\tau}:=b(S)$ stand for the vector associated with a cyclic
set $S\subset\Sigma(\tau)$ so that (\ref{recwords}) holds. Let
$$b^{(n)}_{\pi}=A^{(k_n+k)}b_{\tau}.$$ By
Lemma~\ref{ww}, $b^{(n)}_{\pi}=b(S')$ for some $S'\in\Sigma(\pi)$.

\begin{remark}\label{liczpop}
Suppose that $x_r\in I_r^{(k_n)}\subset
\bigcap_{i=0}^{K_r}T^{-i}_{\rho^{(k_n)},\tau}\Delta^{(k_n)}_{w^r_i}$
for $r=1,2$. Let us consider the word
$\bar{w}'_r\in\{1,\ldots,m\}^{q_r^{(k_n)}}$ determined by
$T_{\rho,\pi}^ix_r\in\Delta_{w^{'r}_i}$ for $0\leq i<q_r^{(k_n)}$
and $r=1,2$. From (\ref{zmianah}) we have
$l(\bar{w}'_r)=A^{(k_n+k)}l(\bar{w}_r)$ for $r=1,2$. It follows
that
\begin{equation}\label{zmroznica}
l(\bar{w}'_1)-l(\bar{w}'_2)=A^{(k_n+k)}(l(\bar{w}_1)-l(\bar{w}_2))=A^{(k_n+k)}b_{\tau}=b^{(n)}_{\pi}=b(S')
\end{equation}
for some $S'\in\Sigma(\pi)$. Choose $n$ such that
$[0,|\rho^{(k_n)}|)\subset\Delta_1$. Since
$T_{\rho,\pi}^{q_r^{(k_n)}}x_r=T^{K_r}_{\rho^{(k_n)},\tau}x_r\in
\Delta^{(k_n)}_{w^r_0}$, $x_r\in \Delta^{(k_n)}_{w^r_0}$ and
$\Delta^{(k_n)}_{w^r_0}\subset \Delta_1$, we conclude that
$\bar{w}'_r$ is a recurrence word of $T_{\rho,\pi}$ for $r=1,2$.
\end{remark}
This gives the following conclusion.

\begin{lemma}\label{przeniesienie}
Suppose that $(\lambda,\tau)\in\Lambda_m\times{S}_m^0$ is a pair
such that the IET $T_{\lambda,\tau}$  satisfies the IDOC and
$T_{\lambda,\tau}$ has recurrence words
$\bar{w}_1,\;\bar{w}_2\in\mathcal{L}(T_{\lambda,\tau})$ such that
$l(\bar{w}_1)-l(\bar{w}_2)=b(S)$ for some $S\in\Sigma(\tau)$. If
$\pi\in\mathcal{R}(\tau)$ then for almost every $\rho\in\Lambda_m$
the IET $T_{\rho,\pi}$ has recurrence words
$\bar{w}'_1,\;\bar{w}'_2\in\mathcal{L}(T_{\rho,\pi})$ such that
$l(\bar{w}'_1)-l(\bar{w}'_2)=b(S')$ for some $S'\in\Sigma(\pi)$.
\bez
\end{lemma}

For any $h=(h_1,\ldots,h_m)\in\R_+^m$ denote by
$f_h:[0,1)\rightarrow\mathbb{N}$ the step function
$f_h=\sum_{i=1}^{m}h_i\chi_{\Delta_i^{\rho}}$.

\begin{lemma}\label{stalykocykl} For every $x\in C_{r}^{(k_n)}$, $r=1,2$ we have
$$f_h^{(q_r^{(k_n)})}(x)=h\cdot A^{(k_n+k)}\cdot l(\bar{w}_r).$$
\end{lemma}

\begin{proof}
First suppose that $x\in I_r^{(k_n)}\subset
\bigcap_{i=0}^{K_r}T^{-i}_{\rho^{(k_n)},\tau}\Delta^{(k_n)}_{w^r_i}$.
Since $T_{\rho,\pi}^ix\in\Delta_{w^{'r}_i}$ for $0\leq
i<q_r^{(k_n)}$, by the definition of $f_h$,
\[f_h^{(q^{(k_n)}_r)}(x)=\sum_{i=0}^{q_r^{(k_n)}-1}h_{w^{'r}_i}=h\cdot l(\bar{w}_r').\]
In view of Remark~\ref{liczpop},
\[f_h^{(q^{(k_n)}_r)}(x)=h\cdot l(\bar{w}_r')=h\cdot A^{(k_n+k)}\cdot l(\bar{w}_r).\]

If $x\in C_{r}^{(k_n)}$ then there exist $y\in I_r^{(k_n)}$ and
$0\leq i<h_{w^r_0}^{(k_n)}$ such that $x=T^i y$. In view of the
proof of Lemma~\ref{12}, $x,\ T^{q_r^{(k_n)}}x\in
T^i\Delta^{(k_n)}_{w^r_0}$, and hence $x,\ T^{q_r^{(k_n)}}x\in
\Delta_j$ for some $1\leq j\leq m$. Thus
$f_h(T^{q_r^{(k_n)}}x)=f_h(x)$ for every $x\in C_{r}^{(k_n)}$. It
follows that
\begin{eqnarray*}f_h^{(q^{(k_n)}_r)}(x)&=&f_h^{(q^{(k_n)}_r)}(T^i
y)=f_h^{(q^{(k_n)}_r)}(y)+f_h^{(i)}(T^{(q^{(k_n)}_r)}y)-f_h^{(i)}(y)\\
&=&f_h^{(q^{(k_n)}_r)}(y)+\sum_{l=0}^{i-1}(f_h(T^{(q^{(k_n)}_r)}T^ly)-f_h(T^ly))\\
&=&f_h^{(q^{(k_n)}_r)}(y)=h\cdot A^{(k_n+k)}\cdot l(\bar{w}_r).
\end{eqnarray*}
\end{proof}

Setting $a_r^{(k_n)}:=h\cdot A^{(k_n+k)}\cdot l(\bar{w}_r)$ we
obtain
$$\frac{1}{\mu(C_{r}^{(k_n)})}\left(\left.\left(f_h^{(q_r^{(k_n)})}-a_r^{(k_n)}\right)\right|_{C_{r}^{(k_n)}}\right)_{*}\left(\mu|_{C_{r}^{(k_n)}}\right)=
\delta_0\text{ for } r=1,2.$$

Now we apply Proposition~\ref{twozbdoc} together with
Lemmas~\ref{miara} and \ref{12} to obtain the following result.

\begin{theorem}\label{zb}
For $r=1,2$ there exist $\alpha_r>0$ and $\Phi_r\in
\mathcal{J}(T^{f_h})$ such that
\[(T^{f_h})_{a_r^{(k_n)}}\rightarrow\alpha_r
Id+(1-\alpha_r)\Phi_r\] in weak operator topology as
$n\rightarrow\infty$. \bez
\end{theorem}

\begin{theorem} \label{vvv}
Assume that $T$ is ergodic, $h\notin H(\pi)$ and $\theta_h=h\cdot
b_{\pi}^{(n)}\neq 0$. Then there exist a sequence $\{a_n\}$,
$a_n\to+\infty$, positive numbers $\alpha_1,\alpha_2$ with
$\alpha_1+\alpha_2\leq1$, and $\Phi\in \mathcal{J}(T^{f_h})$ such
that
$$(T^{f_h})_{a_n}\rightarrow\alpha_1Id+\alpha_2(T^{f_h})_{\theta_h}+(1-\alpha_1-\alpha_2)\Phi$$
  in weak operator topology as  $n\rightarrow\infty$.
Hence the special flow $T^{f_h}$ is disjoint from all weakly
mixing ELF flows.
\end{theorem}
\begin{proof}  By the definitions of $b_\pi^{(n)}$, $\theta_h$ and
(\ref{zmroznica}),
\begin{eqnarray*}
a_1^{(k_n)}-a_2^{(k_n)}&=&h\cdot A^{(k_n+k)}\cdot
l(\bar{w}_1)-h\cdot A^{(k_n+k)}\cdot l(\bar{w}_2)\\& =&h\cdot
A^{(k_n+k)}\cdot(l(\bar{w}_1)-l(\bar{w}_2))\\
&=&h\cdot A^{(k_n+k)}\cdot b_{\tau}=h\cdot
b^{(n)}_{\pi}=\theta_h\neq 0.
\end{eqnarray*} Set $a_n:=a_1^{(k_n)}$.  By Theorem \ref{zb}, it
follows that
$$(T^{f_h})_{a_n}\rightarrow\alpha_1 Id+(1-\alpha_1)\Phi_1$$
and
$$(T^{f_h})_{a_n}\rightarrow\alpha_2(T^{f_h})_{\theta_h}+(1-\alpha_2)(T^{f_h})_{\theta_h}
\circ\Phi_2.$$ Since $T^{f_h}$ is ergodic, by Lemma~\ref{alfy},
$\alpha_1+\alpha_2\leq 1$ and there exists
$\Phi\in\mathcal{J}(T^{f_h})$ for which
$$(T^{f_h})_{a_n}\rightarrow\alpha_1Id+\alpha_2(T^{f_h})_{\theta_h}+(1-\alpha_1-\alpha_2)\Phi.$$
Therefore
\[(T^{f_h})_{a_n}\rightarrow(\alpha_1+\alpha_2)\int_\R(T^f)_{t}\,dP(t)+(1-\alpha_1-\alpha_2)\Phi,
\] where
$P=\frac{\alpha_1}{\alpha_1+\alpha_2}\delta_0+\frac{\alpha_2}{\alpha_1+\alpha_2}\delta_{\theta_h}$.
Now it suffices to apply Proposition~\ref{proprozlflow} to
complete the proof.
\end{proof}

\begin{corollary}\label{corflow}
If $\pi\in S^0_m$ and $\#\Sigma(\pi)\neq 1$ then for almost every
$\lambda\in\R^m_+$ and $h\in\R^m_+$ such that $h\cdot b(S)\neq0$
for all $S\in\Sigma(\pi)$ the special flow $T^{f_h}_{\lambda,\pi}$
is disjoint from weakly mixing ELF flows. In particular, for a.e.
$\lambda\in\R^m_+$ and a.e. $h\in\R^m_+$ the special flow
$T^{f_h}_{\lambda,\pi}$ is disjoint from weakly mixing ELF
flows.\bez
\end{corollary}

\begin{theorem}\label{roz_z_5}
If
$(\rho,\pi)\in\mathcal{A}\cap(\Lambda_m\times\widetilde{S}^0_m)$
then the IET $T=T_{\rho,\pi}:[0,1)\rightarrow[0,1)$ is weakly
mixing and disjoint from all ELF automorphisms.
\end{theorem}
\begin{proof}  Let
$h=(1,\ldots,1)$, hence $f_h\equiv 1$. Since
$\pi\in\widetilde{S}_m^0$, it follows that
$\theta_1=(1,\ldots,1)\cdot b(S)=|b(S)|=\pm1$ and hence $h\notin
H(\pi)$. Without loss of generality we can assume that
$\theta_1=1$. By Theorem \ref{vvv}, there exists a sequence
$\{a_n\}$ of natural numbers ($a_n=h\cdot A^{(k_n+k)}\cdot
l(\bar{w}_1)\in\N$), positive numbers $\alpha_1,\alpha_2$ with
$\alpha_1+\alpha_2\leq1$, and $\Phi\in\mathcal{J}(T^{f_1})$ such
that
$$(T^{f_1}_1)^{a_n}\rightarrow\alpha_1 Id_{[0,1)\times[0,1)}+\alpha_2T^{f_1}_1+(1-\alpha_1-\alpha_2)\Phi.$$
Since $T^{f_1}_1=(T^1)_1= T\otimes Id$ ($Id:=Id_{[0,1)}$), we have
\begin{equation}\label{jeden}
(T^{a_n}\otimes Id)\rightarrow\alpha_1(Id\otimes
Id)+\alpha_2(T\otimes Id)+(1-\alpha_1-\alpha_2)\Phi,
\end{equation}
and hence
\begin{equation}\label{pierd}
(T^{a_n}-\alpha_1 Id-\alpha_2T)\otimes
Id\rightarrow(1-\alpha_1-\alpha_2)\Phi.
\end{equation}

Suppose that $\alpha_1+\alpha_2<1$. If $\alpha_1+\alpha_2=1$ then
the last member of (\ref{jeden}) vanishes and the proof becomes
even easier.

Consider the subspace
$$H=L^2([0,1)\times[0,1),\mathcal{B}_{[0,1)}\otimes\{\emptyset,[0,1)\},\mu\otimes\mu)\subset
L^2([0,1)\times[0,1),\mathcal{B}_{[0,1)}\otimes\mathcal{B}_{[0,1)},\mu\otimes\mu).$$
The space $H$ can be identified with
$L^2([0,1),\mathcal{B}_{[0,1)},\mu)$ via the map
\[L^2([0,1),\mathcal{B}_{[0,1)},\mu)\ni f\mapsto\overline{f}\in H,\;\;\overline{f}(x,y)=f(x).\]
Since $H$ is closed and $(T^{a_n}-\alpha_1 Id-\alpha_2T)\otimes
Id$-invariant, by (\ref{pierd}), it follows that $H$ is also
$\Phi$-invariant. Denote by
$\Phi':L^2([0,1),\mathcal{B},\mu)\rightarrow
L^2([0,1),\mathcal{B},\mu)$ the restriction of $\Phi$ to $H$, more
precisely, $\Phi'$ is determines by
$\overline{\Phi'(f)}=\Phi(\overline{f})$ for $f\in
L^2([0,1),\mathcal{B},\mu)$. In view of (\ref{jeden}), it follows
that
\begin{equation}\label{dowielo}
T^{a_n}\rightarrow\alpha_1
Id+\alpha_2T+(1-\alpha_1-\alpha_2)\Phi'.
\end{equation}
 Since $\Phi$ is a Markov  operator,
$\overline{\Phi'(f)}=\Phi(\overline{f})\geq 0$ for every $f\geq
0$. Moreover,
$\overline{\Phi'(1)}=\Phi(\overline{1})=\overline{1}$ and
\[\int_{[0,1)}\Phi'(f)\,d\mu=\int_{[0,1)\times[0,1)}\Phi(\overline{f})\,d\mu\otimes\mu=\int_{[0,1)\times[0,1)}\overline{f}\,d\mu\otimes\mu=\int_{[0,1)}f\,d\mu,\]
and hence $\Phi'$ is a Markov operator. As
$\Phi\in\mathcal{J}(T^{f_1})$, the operators $\Phi$ and $T\otimes
Id$ commute. It follows that
\[\overline{\Phi'\circ T(f)}=\Phi\circ (T\otimes Id)(\overline{f})=(T\otimes Id)\circ\Phi (\overline{f})=\overline{T\circ \Phi'(f)}.\]
Therefore $\Phi'\in \mathcal{J}(T)$.

Finally we will show that $T$ is weakly mixing. Suppose that
$fT=e^{ia}f$ for some $f\in L^2([0,1))$ with $\|f\|=1$ and
$a\in\R$. From (\ref{dowielo}),
\begin{eqnarray*}1&=&\|f\|^2=|\langle
f,fT^{a_n}\rangle|=|\alpha_1\langle f,f\rangle+\alpha_2\langle
f,fT\rangle+ (1-\alpha_1-\alpha_2)\langle f,\Phi'f\rangle|\\
&\leq&|\alpha_1+\alpha_2e^{ia}|+(1-\alpha_1-\alpha_2)|\langle
f,\Phi'f\rangle|\leq|\alpha_1+\alpha_2e^{ia}|+(1-\alpha_1-\alpha_2).
\end{eqnarray*}
It follows that
\[1\leq\left|\frac{\alpha_1}{\alpha_1+\alpha_2}+\frac{\alpha_2}{\alpha_1+\alpha_2}e^{ia}\right|,\]
thus $e^{ia}=1$, which proves the weak mixing of $T$. Now we can
apply  Proposition~\ref{proprozaut} to complete the proof.
\end{proof}

In view of Remark~\ref{full}  we obtain the following conclusion.

\begin{theorem}\label{fuldis}
Suppose that $(\lambda,\tau)\in\Lambda_m\times\widetilde{S}_m^0$
is a pair such that the IET $T_{\lambda,\tau}$  satisfies the IDOC
and $T_{\lambda,\tau}$ has recurrence words
$\bar{w}_1,\;\bar{w}_2\in\mathcal{L}(T_{\lambda,\tau})$ such that
$l(\bar{w}_1)-l(\bar{w}_2)=b(S)$ for some $S\in\Sigma(\pi)$. If
$\pi\in\mathcal{R}(\tau)\cap\widetilde{S}_m^0$ then for almost
every $\rho\in\Lambda_m$ the IET $T_{\rho,\pi}$ is disjoint from
ELF automorphisms.\bez
\end{theorem}

\section{An example of $5$-IET}\label{secprzyklad}
\label{example} In this subsection we give an explicit example of
$5$-IET of periodic type which fulfills the hypothesis of
Theorem~\ref{fuldis}. This example gives also an example of IET
which is disjoint from ELF automorphisms. To find IETs of periodic
type  we  use a method introduced in \cite{Si-Ul}. This method is
based on searching for some closed paths in Rauzy graphs.

Let $c=c_1\ldots c_n$ be a word over the alphabet $\{a,b\}$.
Denote by  $c(\tau)$ the path of length $n$ in the Rauzy class
$\mathcal{R}(\tau)$ starting from the permutation $\tau$ then we
apply consecutively operations $c_1,\ldots, c_n$. Suppose that
$c(\tau)$ is a closed path, i.e.\ $\tau=c_n\circ\ldots\circ
c_1(\tau)$, and
$$A(c(\tau))=A(c_1,\tau)A(c_2,c_1(\tau))\ldots
A(c_n,c_{n-1}\circ\ldots\circ c_1(\tau))$$ is a primitive matrix.
A method of searching for such paths was described in
\cite{Ma-Mo-Yo}. Let $\theta>1$ stand for the Perron-Frobenius
eigenvalue of $A(c(\tau))$ and let $\lambda\in\R^m_+$ be a right
Perron-Frobenius eigenvector. Since
\[\theta\lambda=A(c_1,\tau)A(c_2,c_1(\tau))\ldots
A(c_n,c_{n-1}\circ\ldots\circ c_1(\tau))\lambda,\] we conclude
that $c(J^{k-1}(\lambda,\tau))=c_k$ and
\[J^{k}(\lambda,\tau)=(A(c_k,c_{k-1}\circ\ldots\circ c_1(\tau))^{-1}\ldots A(c_2,c_1(\tau))^{-1}A(c_1,\tau)^{-1}\lambda,c_k\circ\ldots\circ
c_1(\tau))\] for $1\leq k\leq n$. It follows that
$J^{n}(\lambda,\tau)=(\theta^{-1}\lambda,\tau)$, and hence
$P^{n}(\lambda,\tau)=(\lambda,\tau)$, which shows that
$T_{\lambda,\tau}$ has periodic type.

 Let us consider the
permutation $\tau_5^{sym}$. In this case
$$\eta_{\tau_5^{sym}}=\begin{pmatrix}0&1&2&3&4&5\\4&5&0&1&2&3\end{pmatrix}$$
and $\eta_{\tau_5^{sym}}$ has two cyclic sets $S_0=\{0,2,4\}$ and
$S_1=\{1,2,5\}$ for which  $b(S_0)=(1,-1,1,-1,1)$ and
$b(S_1)=(-1,1,-1,1,-1)$. Therefore
$\tau_5^{sym}\in\widetilde{S}_5^0$.

\begin{figure}[htp]
\begin{center}
\includegraphics[width=10cm]{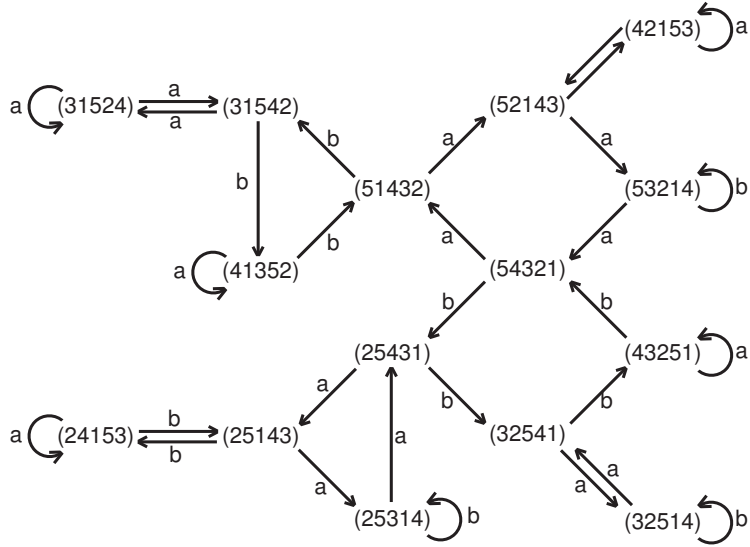}
\end{center}
\caption{Rauzy class $\mathcal{R}(\tau_5^{sym})$}\label{rysunek}
\end{figure}

Next consider the closed path $c(\tau_5^{sym})$, where $c=
bbaababaaaba$, in the Rauzy class $\mathcal{R}(\tau_5^{sym})$ (see
Figure~\ref{rysunek}). Here
$$A=A(c(\tau_5^{sym}))=\begin{pmatrix}
  1 & 1 & 1 & 1 & 1 \\
  1 & 2 & 0 & 0 & 0 \\
  0 & 0 & 2 & 3 & 2 \\
  0 & 0 & 0 & 2 & 1 \\
  2 & 3 & 2 & 2 & 2
\end{pmatrix},\;\;B=A^2=\begin{pmatrix}
  4 & 6 & 5 & 8 & 6 \\
  3 & 5 & 1 & 1 & 1 \\
  4 & 6 & 8 & 16 & 11 \\
  2 & 3 & 2 & 6 & 4 \\
  9 & 14 & 10 & 16 & 12
\end{pmatrix},$$
$\theta=2+\frac{1}{2}\sqrt{3}+\frac{1}{2}\sqrt{15+8\sqrt{3}}\approx
5.55$ is the Perron-Frobenius  eigenvalue of $A$ and
\begin{eqnarray*}
\lambda&=&\left(\sqrt{3},\frac{3}{2}-\sqrt{3}+\sqrt{15+8\sqrt{3}}-\frac{1}{2}\sqrt{3}\sqrt{15+8\sqrt{3}},\right.\\
&&\left.-1+\frac{3}{2}\sqrt{3}-\frac{3}{2}\sqrt{15+8\sqrt{3}}+\sqrt{3}\sqrt{15+8\sqrt{3}},1,\frac{1}{2}\sqrt{3}+\frac{1}{2}\sqrt{15+8\sqrt{3}}\right)
\end{eqnarray*}
is a right Perron-Frobenius eigenvector. Therefore the IET
$T:=T_{\lambda,\tau_5^{sym}}:[0,|\lambda|)\rightarrow[0,|\lambda|)$
has  periodic type, hence is minimal. Moreover,
$T_{\lambda,\tau_5^{sym}}$ is weakly mixing. This is a consequence
of Corollary 1 in \cite{Si-Ul} because the characteristic
polynomial of $A$ is equal to $p(x)=(x-1)(x^4-8x^3+15x^2-8x+1)$
and the Galois groups of $p$ has five elements.

In view of Remark~\ref{isosub}, every IET of periodic type is
isomorphic to a substitution dynamical system
$S_\sigma:X_\sigma\to X_\sigma$. For the IET
$T_{\lambda,\tau_5^{sym}}$ the substitution $\sigma$ is defined
over the alphabet $\{1,\ldots,5\}$ and
$$\sigma(1)=1525,\ \sigma(2)=152525,\ \sigma(3)=15335,\ \sigma(4)=15343435,\ \sigma(5)=153435.$$
Since every word $\sigma(i)$ for $1\leq i\leq 5$ starts with the
symbol $1$, if $\bar{w}\in\mathcal{L}_\sigma$ then
$\sigma(\bar{w})\in\mathcal{L}_\sigma$ is a recurrence word for
$\sigma$. As
$\mathcal{L}_\sigma=\mathcal{L}(T_{\lambda,\tau^{sym}_5})$ it
follows that $\sigma(\bar{w})$ is a  recurrence word for the IET
$T$.

Let $u\in\{1,\ldots,5\}^\N$ stand for the unique fixed sequence
for
$\sigma:\{1,\ldots,5\}^{\mathbb{N}}\rightarrow\{1,\ldots,5\}^{\mathbb{N}}$
starting from the symbol $1$, i.e.
$$u=15\textbf{251534}35152525153435152515343\textbf{5153351}534343515335153435152\ldots$$
Note that $251534,\;5153351\in\mathcal{L}_\sigma$ and
$l(251534)-l(5153351)=(-1,1,-1,1,-1)$. Let
$\bar{w}_1=\sigma(251534)$ and $\bar{w}_2=\sigma(5153351)$. Then
$\bar{w}_1,\;\bar{w}_2$ are recurrence words for $T$ and
\begin{eqnarray*}
l(\bar{w}_1)-l(\bar{w}_2)&=&Al(251534)-Al(5153351)=A(-1,1,-1,1,-1)\\&=&(-1,1,-1,1,-1)=b(S_1).
\end{eqnarray*}

The closed path $c(\tau^{sym}_5)$, the substitution $\sigma$ and
the recurrence words $\bar{w}_1$, $\bar{w}_2$ were found with the
help of Maple.

Now can we apply Theorems~\ref{roz_z_5} and  \ref{fuldis} to have
the following result.
\begin{theorem}\label{piec} If
$\pi\in\mathcal{R}(\tau^{sym}_5)\cap\widetilde{S}_5^0$ then for
almost every $\rho\in\Lambda_5$ the IET $T_{\rho,\pi}$ is disjoint
from ELF automorphisms. Moreover, the periodic type IET
$T_{\lambda,\tau^{sym}_5}$ is also disjoint from ELF
automorphisms.
\end{theorem}

\begin{proof}
The first part of the theorem is a simple consequence of
Theorem~\ref{fuldis}.

To prove the second part let us consider the set
$$\mathcal{M}=\left\{\left(\frac{B\lambda^{\varepsilon}}{|B\lambda^{\varepsilon}|},\tau^{sym}_5\right):\
(\lambda^{\varepsilon},\tau^{sym}_5)\in\mathcal{K}\right\}$$ with
$B=A^2=A(cc(\tau^{sym}_5))=A^{(24)}(\lambda,\tau^{sym}_5)$. (See
the beginning of Subsection~\ref{zaburzenia} for the definition of
$\mathcal{K}$). Since $B\lambda=\theta\lambda$ and
$P^{24}(\lambda,\tau^{sym}_5)=(\lambda,\tau^{sym}_5)$ we have
$P^{24k}(\lambda,\tau^{sym}_5)\in\mathcal{M}$ for every natural
$k$. Consequently, $(\lambda,\tau^{sym}_5)\in\mathcal{A}$, which
implies, by Theorem~\ref{roz_z_5}, the disjointness of
$T_{\lambda,\tau^{sym}_5}$ from ELF automorphisms.
\end{proof}

\section{Procedure of reduction}
In this section we describe a procedure which helps us to reduce
the problem of searching recurrence words satisfying
(\ref{recwords}) to a smaller number of intervals.

Assume that $m\geq 5$ is odd. Let $$\tau_m^{sym}=\begin{pmatrix}
    1&2&\ldots&m-1&m\\
    m&m-1&\ldots&2&1\\
        \end{pmatrix}$$
and $$\tau_m=\begin{pmatrix}
    1&2&3&4&\ldots&m-3&m-2&m-1&m\\
    m-1&1&m-2&m-3&\ldots&4&3&m&2\\
        \end{pmatrix}.$$
\begin{remark}\label{vectors}The cyclic sets of
$\eta_{\tau_m^{sym}}$ are of the form
$$S^m_0=\{0,2,4,\ldots,m-3,m-1\},\ S^m_1=\{1,3,\ldots,m-2,m\}$$ with associated vectors
$$b(S^m_0)=(1,-1,1,\ldots,1,-1,1),\ b(S^m_1)=(-1,1,-1,\ldots,-1,1,-1).$$ The cyclic sets of
$\eta_{\tau_m}$ are
$$Q^m_0=\{0,1,3,\ldots,m-4,m-2\},\ Q^m_1=\{2,4,\ldots,m-3,m-1,m\}$$
with
\begin{equation}\label{qzero}
b(Q^m_0)=(0,1,-1,1,\ldots,1,-1,1,0),\
b(Q^m_1)=(0,-1,1,-1,\ldots,-1,1,-1,0).
\end{equation} Hence
$\tau_m,\tau_m^{sym}\in\widetilde{S}_m^0$.
\end{remark}

\begin{lemma}\label{moj}
For any odd $m \geq 5$ we have
$\tau_m\in\mathcal{R}(\tau_m^{sym})$.
\end{lemma}
\begin{proof} First note that \begin{displaymath}
        a\tau_m^{sym}(i)=\left\{
        \begin{array}{lll}
        m&\text{ if } & i=1, \\
        1&\text{ if } & i=2, \\
        m-i+2&\text{ if } & 3\leq i\leq m.\\
        \end{array}
        \right.
        \end{displaymath}
Set $b^k=\overbrace{b\circ\ldots\circ b}^{k}$ for $ k\geq 1$ and
$b^0=Id$. Thus for $1\leq k\leq m-3$,
\begin{displaymath}
        b^ka\tau_m^{sym}(i)=\left\{
        \begin{array}{ll}
        1 &\text{ if } i=2, \\
        2 &\text{ if } i=m, \\
        3 &\text{ if } i=(b^{k-1}a\tau_m^{sym})^{-1}(m), \\
        b^{k-1}a\tau_m^{sym}(i)+1 & \mbox{ otherwise}.\\
        \end{array}
        \right.
        \end{displaymath}
It follows that $b^{m-3}a\tau_m^{sym}=\tau_m$.
\end{proof}

\begin{theorem}\label{koncowe}
For every odd $m\geq 5$ there exists an IET
$T_{\lambda,\tau_m^{sym}}$ fulfilling the IDOC which has two
recurrence words $\bar{w}_1$, $\bar{w}_2$ such that
$l(\bar{w}_1)-l(\bar{w}_2)=b(S)$ for some
$S\in\Sigma(\tau_m^{sym})$.

If $\pi\in\widetilde{S}_m^0\cap \mathcal{R}(\tau_m^{sym})$ then
for almost all $\rho\in\Lambda_m$, $T_{\rho,\pi}$ is disjoint from
ELF automorphisms.
\end{theorem}
\begin{proof}
The proof of the first part is by induction on $m$. In the base
step at $m=5$ we use the example from Section~\ref{secprzyklad}.

In the inductive step suppose that $m\geq 7$ and there exists an
IET $T_{\lambda,\tau_{m-2}^{sym}}$ ($|\lambda|=1$) with the IDOC
which has two recurrence words $\bar{w}_1$,
$\bar{w}_2\in\{1,\ldots,m-2\}^*$ of length $K_1$, $K_2$
respectively,  such that $l(\bar{w}_1)-l(\bar{w}_2)=b(S^{m-2}_0)$,
where $b(S^{m-2}_0)=(1,-1,1,\ldots,-1,1)$. Set
\[\widetilde{\lambda}=(0,\lambda_1,\ldots,\lambda_{m-2},0)\in\R^m_+\]
and
\[\widetilde{w}_r=w_1^r+1\ldots w_{K_r}^r+1\in\{1,\ldots,m\}^{K_r}\text{ for }r=1,2.\]
Since $T_{\widetilde{\lambda},\tau_m}\equiv
T_{\lambda,\tau_{m-2}^{sym}}$, the $m$-IET
$T_{\widetilde{\lambda},\tau_m}$ satisfies the weak IDOC. If
$\widetilde{\Delta}_1,\ldots,\widetilde{\Delta}_m$ are intervals
exchanged by $T_{\widetilde{\lambda},\tau_m}$ then
$T^k_{\widetilde{\lambda},\tau_m}x\in\widetilde{\Delta}_{j+1}$ if
and only if $T^k_{{\lambda},\tau^{sym}_{m-2}}x\in{\Delta}_{j}$. It
follows that $\widetilde{w}_1$, $\widetilde{w}_2$ are recurrence
words for $T_{\widetilde{\lambda},\tau_m}$. Moreover,
\[l(\widetilde{w}_r)=(0,l(\bar{w}_r),0)\text{ for }r=1,2,\]
and hence
\[l(\widetilde{w}_1)-l(\widetilde{w}_2)=(0,l(\bar{w}_1)-l(\bar{w}_2),0)=(0,b(S^{m-2}_0),0).\]
In view of (\ref{qzero}), we have
\[l(\widetilde{w}_1)-l(\widetilde{w}_2)=(0,b(S^{m-2}_0),0)=b(Q^m_0).\]
 Let $K=\max(K_1,K_2)+1$ and take
$\vep=\vep_{T_{\widetilde{\lambda},\tau_m},K}$ (see
(\ref{zalozenie}) in Subsection~\ref{zaburzenia}). Since
$$\mathcal{K}=\{(\lambda',\tau_m)\in\Lambda_m\times\{\tau_m\}:\
|\lambda'|=1,\ d(\lambda',\widetilde{\lambda})<\vep\}$$ has
positive Lebesgue measure on $\Lambda_m\times\{\tau_m\}$, we can
find  $(\lambda',\tau_m)\in\mathcal{K}$ such that
$T_{\lambda',\tau_m}$ satisfies the IDOC. By
Lemma~\ref{iteracje_3}, $\widetilde{w}_1$, $\widetilde{w}_2$ are
recurrence words for $T_{\lambda',\tau_m}$. In view of
Lemma~\ref{przeniesienie}, there exists
$(\widehat{\lambda},\tau^{sym}_m)\in\Lambda_m\times\{\tau^{sym}_m\}$
such that the IET $T_{\widehat{\lambda},\tau_m}$ satisfies the
IDOC and it has two recurrent words $\widehat{w}_1$ and
$\widehat{w}_2$ such that $l(\widehat{w}_1)-l(\widehat{w}_2)=
b(S)$ for some $S\in\Sigma(\tau^{sym}_m)$.

Now we apply Theorem~\ref{fuldis} to complete the proof of the
second part.
\end{proof}

\end{document}